\documentclass[12pt]{amsart}

\usepackage{amssymb}
\usepackage{graphicx}
\usepackage{amsmath}
\numberwithin{equation}{section}
\usepackage[all]{xy}

\newtheorem{Theorem}{Theorem}[section]
\newtheorem{Proposition}[Theorem]{Proposition}
\newtheorem{Lemma}[Theorem]{Lemma}
\newtheorem{Corollary}[Theorem]{Corollary}
\newtheorem{Definition}[Theorem]{Definition}
\newtheorem{Remark}[Theorem]{Remark}

\numberwithin{equation}{section}

\begin{document}

\baselineskip=16pt

\title{A note on rank two stable  bundles over surfaces}

\author{Graciela Reyes-Ahumada}

\address{CONACyT - U. A. Matem\'aticas, U. Aut\'onoma de
Zacatecas
\newline  Calzada Solidaridad entronque Paseo a la
Bufa, \newline C.P. 98000, Zacatecas, Zac. M\'exico.}

\email{grace@cimat.mx}

\author{L. Roa-Leguizam\'on}

\address{Instituto de F\'{\i}sica y Matem\'aticas \newline Universidad
Michoacana de San Nicol\'as de Hidalgo \newline Edificio C3,
Ciudad Universitaria \newline C.P.58040 Morelia, Mich. M\'exico.}

\email{leonardo.roa@cimat.mx}

\author{H. Torres-L\'opez}

\address{CONACyT - U. A. Matem\'aticas, U. Aut\'onoma de
Zacatecas
\newline  Calzada Solidaridad entronque Paseo a la
Bufa, \newline C.P. 98000, Zacatecas, Zac. M\'exico.}

\email{hugo@cimat.mx}

\thanks{The second author acknowledges the financial support of Programa para el Desarrollo Profesional Docente, para el Tipo Superior (PRODEP), clave UMSNH-CA-165.}

\subjclass[2010]{}


\date{\today}

\maketitle

\begin{abstract}
Let $\pi: X \longrightarrow C$ be a fibration with reduced fibers over a curve $C$ and consider a polarization $H$ on the surface $X$. Let $E$ be a stable vector bundle of rank $2$ on $C$. We prove that the pullback $\pi^*E$ is a $H-$stable bundle over $X$. This result allows us to relate the corresponding moduli spaces of stable bundles $\mathcal{M}_C(2,d)$ and $\mathcal{M}_{X,H}(2,df,0)$ through an injective morphism. We study the induced morphism at the level of Brill-Noether loci to construct examples of Brill-Noether loci on fibered surfaces. Results concerning the emptiness of Brill-Noether loci follow as a consequence of a generalization
of Clifford's Theorem for rank two bundles on surfaces.
\end{abstract}

\section{Introduction}

Let $C$ be a smooth irreducible complex projective curve of genus $g$. A fibration over $C$ is a surjective morphism $\pi:X\longrightarrow C$ from a  projective nonsingular surface $X$ with connected fibres. Consider the case when $\pi: X \longrightarrow C$ is a ruled surface. Let $E$ be a stable vector bundle of rank two on $C$. It is shown in \cite[Proposition 3.4]{Takemoto} that for any ample line bundle $H$ on $X$ the pullback $\pi^*E$ is a $H-$stable bundle on $X$. This result has been generalized by S. Misra to higher rank bundles to study stable Higgs bundles on ruled surfaces (cf. \cite[Corollary 4.2]{misra}). In this case, there is an isomorphism between the moduli space $\mathcal{M}_C(r,d)$ of stable rank $r$ vector bundles  with degree $d$ on $C$ and the moduli space $\mathcal{M}_{X,H}(r,df,0)$ of $H-$stable rank $r$ vector bundles with fixed Chern classes $c_1 = df$ where $f$ denotes the class of a fiber of $\pi$ and $c_2 =0$ on the ruled surface $X$ (cf. \cite[Theorem 5.1]{misra}). When $\pi: X \longrightarrow C$ is a non-isotrivial elliptic fibration with $\chi(\mathcal{O}_X)>0$, there is also an isomorphism between the  moduli spaces $\mathcal{M}_C(r,d)$ and  $\mathcal{M}_{X,H}(r,df,0)$ (see \cite{Bauer} and \cite{Varma}).

In this paper we aim to generalize these results to fibrations with reduced fibers  in the case of rank two bundles. Let $\pi: X \longrightarrow C$ be a fibration with reduced fibers and consider an ample line bundle $H$ over $X$. Let $E$ be a stable vector bundle of rank $2$ on $C$, we prove in Theorem \ref{TheoremA} that the pullback $\pi^*E$ is a $H-$stable rank $2$ vector bundle over $X$. This result allows us to relate the corresponding moduli spaces over the curve and the surface, $\mathcal{M}_C(2,d)$ and $\mathcal{M}_{X,H}(2,df,0)$ respectively. More precisely, we prove the following 

\textbf{Theorem 3.7} Let $\pi:X\to C$ be a fibration with reduced fibers. Then $\pi$ induces an injective morphism of moduli spaces
\begin{gather*}
\pi^{*}:\mathcal{M}_C(2,d) \to  \mathcal{M}_{X,H}(2,df,0)\ \ \ \\
\nonumber E \mapsto  \pi^{*}E.
\end{gather*}


According to the proof of \cite[Theorem 2.3]{laurayrosaBN}, we define  the Brill-Noether locus as  $$W_{X,H}^k(2,c_1,c_2):=\{ E\in \mathcal{M}_{X,H}(2,c_1,c_2) | h^0(E)+h^2(E)\geq k \}$$ parametrizing $H-$stable rank $r$ bundles on a smooth projective surface $X$ with fixed Chern classes $c_1$, $c_2$.
 In the case that the cohomology for every $E\in \mathcal{M}_{X,H}(2,c_1,c_2)$ satisfies $h^2(E)=0$, Costa and Mir\'o-Roig \cite{laurayrosaBN}  computed the expected dimension of a non-empty irreducible component of $W_{X,H}^k(2,c_1,c_2)$ namely $\rho_X(2,c_1,c_2,k)$. 
Consider a fibration $\pi:X\longrightarrow C$ as before and the Brill-Noether locus on $C$,
$$W_C^k(r,d):=\{ E\in \mathcal{M}_C(r,d) | h^0(E)\geq k \}.$$

From the projection formula, the morphism given in Theorem \ref{morfismobrillnoether} induces an injective morphism at the level of Brill-Noether loci $$\pi^{*}:W_C^k(2,d) \to  W_{X,H}^k(2,df,0).$$ We study this morphism to understand the geography of $W_{X,H}^k(2,df,0)$ over the surface. An interesting application of Theorem \ref{morfismobrillnoether} is that it allows to use results on Brill-Noether over curves to determine properties of the locus $W_{X,H}^k(2,df,0)$ on the surface $X$ as follows: If $X$ is a ruled surface, from \cite{misra} the map $\pi^*$ is an isomorphism of moduli spaces for bundles of any rank $r\geq 1$, then it induces an isomorphism of Brill-Noether locus $$W_C^k(r,d) \cong W_{X,H}^k(r,df,0),$$ and the geometry of the locus over the surface $X$ coincides with the one over the curve $C$. When $\pi$ is an elliptic fibration, or when $r=1,2$,  the induced map will be injective. We prove results on non-emptiness of $W_{X,H}^k(r,df,0)$, and by imposing cohomological assumptions we construct examples of Brill-Noether loci on surfaces. Furthermore, we prove a generalization of Clifford's Theorem for rank two bundles on surfaces and as a consequence we prove results concerning the emptiness of the loci $W_{X,H}^k(2,c_1,c_2)$.

The paper is organized as follows: Section $2$ collects a number of classical results, mainly about coherent sheaves on surfaces, that will be subsequently used. Section $3$ is the core of this paper it contains the proof of Theorem \ref{morfismobrillnoether}. In Section $4$ we recall the results of \cite{laurayrosaBN} about the construction of Brill-Noether locus on surfaces and we show some applications of Theorem \ref{morfismobrillnoether} to the study of the geometry of Brill-Noether loci of bundles on fibered surfaces. In section $5$ we prove a generalization of Clifford's Theorem for rank two bundles over surfaces and we show examples where the Brill-Noether loci $W_{X,H}^k(2,c_1,c_2)$ is empty and the expected dimension $\rho_{X,H}(2,c_1,c_2,k)$ is negative.

\textbf{Notation:} We work over the field of complex numbers $\mathbb{C}$.  Given a coherent sheaf $\mathcal{G}$ on a variety $X$ we write $h^i(\mathcal{G})$ to denote the dimension of the $i$-th cohomology group $H^i(X,\mathcal{G})$. The sheaf $K_X$ will denote the canonical sheaf on $X$.

\section{Preliminaries}
This section contains some useful results on vector bundles over surfaces that will be used in the next sections. For detailed treatment of the subject see \cite{Friedman} and \cite{lepotier}.

Let $X$ be a smooth, irreducible, complex, projective variety of dimension $n$ and let $H$ be an ample line bundle over $X$.  Let $\mathcal{G}$ be a torsion free sheaf on $X$ of rank $rk(\mathcal{G})$ with Chern classes $c_i(\mathcal{G}) \in H^{2i}(X,\mathbb{Z})$. The $H-$slope  of $\mathcal{G}$ is defined as the rational number 
\begin{equation*}
\mu_H(\mathcal{G})=\frac{c_1(\mathcal{G}). H^{n-1}}{rk(\mathcal{G})},
\end{equation*}
where  $c_1(\mathcal{G}). H^{n-1}$ is the degree of $\mathcal{G}$ with respect to $H$.

\begin{Definition}
\begin{em}
Let $\mathcal{G}$ be a torsion free coherent sheaf on $X$. We say that $\mathcal{G}$ is $H$-stable (respectively $H$-semistable) if for all coherent subsheaves $\mathcal{F}$ with $0<rk(\mathcal{F})<rk(\mathcal{G})$ we have $\mu_H(\mathcal{F})<\mu_H(\mathcal{G})$ (respectively $\leq$). We call $\mathcal{G}$ unstable if it is not semistable and strictly semistable if it is semistable but not stable.
\end{em}
\end{Definition}
It is well known that the stability of free torsion sheaves over a curve is independent of the polarization. We recall the following remark for the $H$-semistability of vector bundles over a surface $X$:

\begin{Remark}
\begin{em}
Let $V$ be a vector bundle on a surface $X$. We say that $V$ is $H$-(semi)stable if for all subbundles $W\subset V$ with $0<rk(W)<rk(V)$, we have $\mu_H(W)<\mu_H(V)$ ($\leq$ respectively). Indeed, let $W\subset V$ be a proper subsheaf, then $W$ is torsion free and  there exists  the following diagram 
\begin{equation}
 \xymatrix{   W  \ar[r] \ar@{^{(}->}[d]      &  V\ar@{^{(}->}[d] & \\
				 W^{\vee\vee} \ar[r]_{}& V^{\vee\vee}. \\  }
\end{equation}
Since $W^{\vee\vee}$ is a reflexive sheaf and  $c_1( W)=c_1(W^{\vee\vee})$, it follows that $V$ is $H$-semistable (respectively $H$-stable) if for any proper reflexive sheaf $W$ we have $\mu_H(W)\leq\mu_H(V)$ (respectively $<$). Moreover, since  the singular points of reflexive sheaf $W$ have codimension greater than 2 this implies that $W$ is a vector bundle.
\end{em}
\end{Remark}

Moduli spaces of $H$-stable vector bundles with fixed Chern classes $c_1$, $c_2$ on a surface $X$ have been constructed since the '70s by M. Maruyama (cf. \cite{maruyama}). We shall denote the moduli space of $H$-semistable vector bundles of rank $r$ and with fixed Chern classes $c_1,c_2$ on $X$ by $\mathcal{M}^{ss}_{X,H}(r,c_1,c_2)$, and by $\mathcal{M}_{X,H}(2,c_1,c_2)$ the moduli space consisting of stable bundles.

Let $\mathcal{G}$ be a free torsion sheaf on a surface $X$ with Chern classes $c_i$ and rank $n$. The discriminant of $\mathcal{G}$ is the characteristic class
\begin{eqnarray*}
\Delta(\mathcal{G})= 2nc_2-(n-1)c_1^2.
\end{eqnarray*}
The {\bf Bogomolov inequality} states that if $\mathcal{G}$ is $H$-semistable then $\Delta(\mathcal{G})\geq 0.$

The following Proposition gives a relation between the $H-$stability and the discriminant of a vector bundle.

\begin{Proposition} \label{Delta}
If $V$ is a $H$-stable vector bundle of rank $n$ on a smooth algebraic surface $X$ with a fixed polarization $H$ on $X$ and $\Delta(V)=0$ then the stability of the vector bundle $V$ is independent of the chosen polarization.
\end{Proposition}

\begin{proof}
Suppose that there is a polarization $H_1$ such that $V$ is not $H_1$-stable. Then there exists a subbundle $W \subset V$ with $\mu_{H_1}(V)\leq \mu_{H_1}(W).$ Let $W\subset V$ be any subbundle with $\mu_{H_1}(V)\leq \mu_{H_1}(W)$, we can  define a non-negative rational number
\begin{eqnarray*}
\lambda(W) := \frac{\mu_{H_1}(W)-\mu_{H_1}(V)}{\mu_{H}(V)-\mu_{H}(W)}\geq 0, 
\end{eqnarray*}
and an ample divisor $L_W:= H_1+\lambda(W)H$ satisfying  
\begin{eqnarray*}
\mu_{L_W}(W)=\mu_{L_W}(V).
\end{eqnarray*} 
\noindent Notice that if $\lambda(W)\leq \lambda(W_0)$ and $H_0:=H_1+\lambda(W_0)H$ then $\mu_{H_0}(W)\leq \mu_{H_0}(V)$: Indeed we have
\begin{eqnarray*}
 \frac{\mu_{H_1}(W)-\mu_{H_1}(V)}{\mu_{H}(V)-\mu_{H}(W)}=\lambda(W)\leq \lambda(W_0).
\end{eqnarray*}
Then $\mu_{H_1}(W)-\mu_{H_1}(V)\leq \lambda(W_0)(\mu_{H}(V)-\mu_{H}(W))$. Therefore $\mu_{H_0}(W)\leq \mu_{H_0}(V).$ Consider 
\begin{eqnarray*}
A:=\{ \lambda(W) | W\subset V \mbox{ is subbundle  with } \mu_{H_1}(W)\geq \mu_{H_1}(V)\}.
 \end{eqnarray*}
By Grothendieck Theorem (cf. \cite[ Lemma 1.7.9]{huybrechts}), the set $A$ is bounded and we can consider a subbundle $W_0\subset V$ such that $\lambda(W_0)$ is maximal.

\noindent {\bf We claim that:} The vector bundle $V$ is strictly $H_0$-semistable.

\noindent {\it Proof of Claim:} Let $W\subset V$ be a subbundle 
\begin{enumerate}
    \item[(i)] If $\mu_{H_1}(W)<\mu_{H_1}(V)$, then $\mu_{H_0}(W)<\mu_{H_0}(V)$ since $V$ is $H$-stable.
    
    \vspace{.2cm}
    
    \item[(ii)] If $\mu_{H_1}(W)\geq  \mu_{H_1}(V)$, since $\lambda(W_0)$ is maximal in $A$, then $\lambda(W)\leq \lambda(W_0)$ and $\mu_{H_0}(W)\leq \mu_{H_0}(V)$.  
\end{enumerate}
We conclude that $V$ is $H_0$-semistable. Notice that $V$ is not $H_0-$stable because $W_0\subset V$ and $\mu_{H_0}(W_0)=\mu_{H_0}(V).$ This proves the claim.

 Therefore, $W_0$ is $H_0$-semistable and  we have an exact sequence
\begin{equation} \label{3}
    0 \longrightarrow W_0 \longrightarrow V \longrightarrow W_1 \longrightarrow 0
\end{equation}
of torsion free sheaves  with 
\begin{eqnarray*}
\mu_{H_0}(W_0)= \mu_{H_0}(V) = \mu_{H_0}(W_1).
 \end{eqnarray*}
Since $W_0$ and $V$ are $H_0$-semistables torsion free sheaves, it follows that $W_1$ is $H_0$-semistable (cf. \cite[Chapter 4, Lemma 6]{Friedman}). By Bogomolov inequality, we have $\Delta(W_0) \geq 0$  and $ \Delta(W_1) \geq 0.$
Define 
\begin{eqnarray*}
 \mathcal{E}:= (nc_1(W_0)-mc_1(V)),
 \end{eqnarray*}
where $rk \, V = n$ and $rk \, W_0 = m$. 
Since $V$ is stable with respect to the polarization $H$ and strictly semistable with respect to $H_0$, it follows that $\mathcal{E}.H_0 = 0$ and $\mathcal{E}.H >0.$
Therefore $\mathcal{E}$ is not numerically equivalent to zero and  by Hodge Index Theorem $\mathcal{E}^2<0$. On the other hand, from the exact sequence (\ref{3}) we have \begin{eqnarray*}
0= \Delta(V)= \frac{n}{n-m}\Delta(W_0)+ \frac{n}{m}\Delta(W_1)- \frac{\mathcal{E}^2}{m(n-m)}.
\end{eqnarray*}
Since $\Delta(W_0)\geq 0\text{ and } \Delta(W_1) \geq 0$ hence $\mathcal{E}^2 \geq 0 $ which is a contradiction.  We conclude that the stability of the vector bundle $V$ is independent of the polarization. 
\end{proof}

\section{Main Result}

In this section we prove that given a fibration $\pi: X\longrightarrow C$ then the pullback induces an injective morphism from the moduli space $\mathcal{M}_C(2,d)$ of rank $2$ stable vector bundles of degree $d$ on $C$ to the moduli space $\mathcal{M}_{X,H}(2,df,0)$ of $H$-stable  rank $2$ vector bundles with fixed Chern classes $c_1= df$ and $c_2=0$ on $X$ where $f$ is the class of a fiber. We begin this section by recalling the next result (see proof of Corollary 6 in page 54 of \cite{mumford}):

\begin{Lemma}\label{reduced}(cf. \cite{mumford}) Let $B$ be a complete variety (integral separated scheme of finite type over $\mathbb{C}$) and let $L$ be a line bundle over $B$. Then  $L=\mathcal{O}_B$ if and only if $h^0(L)\neq 0\text{ and } h^0(L^{\vee})\neq 0$.
\end{Lemma}

We use the previous Lemma to prove:

\begin{Lemma} \label{principal}
Let $\pi:X\longrightarrow C$ be a fibration with reduced fibers. Let $V$ be a line bundle  such that $V_{|_f}=  \mathcal{O}_f$ for the generic fiber $f$. Then, there exists a line bundle $L$ over $C$ such that $\pi^{*}(L)=V.$
\end{Lemma}
\begin{proof} We define the sets
 $Z:=\{ c\in C| h^0(V|_{\pi^{-1}(c)})\geq 1  \}$ and $W:=\{ c\in C | h^0(V^{\vee}|_{\pi^{-1}(c)})\geq 1 \}.$
By upper-semicontinuity Theorem (cf. \cite{mumford}, Page 50), we have that $Z\text{ and }W$ are closed subsets of $C$. Since $V_{|_f}=\mathcal{O}_f$ for the generic fiber $f$, it follows that $Z=W=C.$ Thus, we get $V_{|_f}=\mathcal{O}_f$ for any fiber $f$ by Lemma \ref{reduced} and consequently $V=\pi^{*}(\pi_{*}(V))$, where $L=\pi_{*}(V)$.
\end{proof}
The following result characterizes semistable bundles of degree zero in terms of their sections:

\begin{Lemma}\label{seccciones0}
Let $E$ be a semistable vector bundle of degree $0$ and rank $r$ over a smooth projective curve $B$. Then $h^0(E)\leq r$ and $h^0(E)=r$ if and only if $E=\oplus \mathcal{O}_B.$
\end{Lemma}

\begin{proof}
If $E$ has rank $r=1$ the statement is clear. Assume that $h^0(E)\neq 0$. If $E$ has rank $r>1$ then there exists an exact sequence 
\begin{eqnarray}\label{sucesiongrado0}
0\rightarrow \mathcal{O}_B \rightarrow E\rightarrow Q\rightarrow 0
\end{eqnarray}
of vector bundles. Since $E$ is a semistable bundle of degree zero it follows that $Q$ is a semistable vector bundle of degree $0$ (cf. \cite[Proposition 5.3.5]{lepotier}). By induction  hypothesis, $h^0(Q)\leq r-1$ and $h^0(Q)=r-1$ if and only if $Q=\oplus \mathcal{O}_B.$ From the exact sequence (\ref{sucesiongrado0}), we get
\begin{eqnarray*}
0\rightarrow H^0(\mathcal{O}_B)\rightarrow H^0(E) \rightarrow H^0(Q)\stackrel{\delta}{\rightarrow} H^1(\mathcal{O}_B).
\end{eqnarray*}
Therefore $h^0(E)\leq 1+h^0(Q)\leq r.$ If $h^0(E)=r$ then we have $h^0(Q)=r-1$ and $Q=\oplus \mathcal{O}_B$ by induction hypothesis. Then $\delta $ is equal to zero and the exact sequence (\ref{sucesiongrado0}) is trivial.
\end{proof}

The following Lemma states a relation between pullback of bundles on the curve and semistable bundles of rank two with trivial restriction:
 \begin{Lemma}\label{restricciones}
  Let $V$ be a $H$-semistable vector bundle of rank $2$ over $X$. Suppose that there exists a vector bundle $E$ over $C$ such that $\pi^{*}(E)=V. $ 
 \begin{itemize}
     \item[(i)]If $V$ is $H$-stable then $E$ is a stable vector bundle.
     \item[(ii)] $V_{|_f}= \mathcal{O}_f \oplus \mathcal{O}_f$ for generic fiber $f$.
 \end{itemize}
 \end{Lemma}
 
 \begin{proof} Let $E_1\subset E$  be a  subbundle. We  want to prove that $\mu(E_1) < \mu(E)$. Since $\pi^{*}(E_1) \subset \pi^{*}(E)=V$ is a subbundle and $V$ is $H$-stable, it follows that
 \begin{eqnarray*}
 \mu(E_1)H.f=\mu_H(\pi^{*}(E_1))<\mu_H(\pi^{*}(E))=\mu(E)H.f.
 \end{eqnarray*}
 Since $H$ is ample and $f$ is a fiber we get $\mu(E_1)< \mu(E)$. Therefore, $E$ is a stable bundle on $C$ and this proves $(i)$. By \cite[Chapter 9, Theorem 18]{Friedman} the bundle $V_{|_f}$ is semistable for the generic fiber $f$ and by projection formula we have $E=\pi_{*}(V)$ since $\pi$ is a fibration. The fiber dimension of $E=\pi_{*}(V)$ at a point $c\in C$ is given as 
 \begin{eqnarray*}
 \text{dim} \pi_{*}(V)|_c=h^0(\pi^{-1}(c),V|_{\pi^{-1}(c)})=h^0(f,V_f)=2=\text{rk}(V),
 \end{eqnarray*}
 where $f=\pi^{-1}(c).$ Finally since $h^0(f,V|_f)=2$ it follows from Lemma \ref{seccciones0} that $V_{|_f}=  \mathcal{O}_f \oplus \mathcal{O}_f$. This proves $(ii).$
\end{proof}
We recall the proof that the pullback of a semistable bundle under a fibration is again semistable:

\begin{Remark}
\label{pullbacksemistable}\emph{(cf. \cite[Theorem 3.3]{misra})
Let $\pi:X\longrightarrow C$ be  a fibration with a fixed polarization $H$ on $X$. Let $E$ be a semistable vector bundle on $C$. Then, the pullback $\pi^{*}(E)$ is $H$-semistable.} \end{Remark}

\begin{proof}[Proof of  Remark \ref{pullbacksemistable}]
Since $H$ is ample there exists a natural $n \in \mathbb{N}$ such that $nH$ is very ample. By Bertini's Theorem there exists a smooth curve $B$ in the linear system  $|nH|$. Consider the composition of morphisms
\begin{eqnarray*}
\phi: B\stackrel{i}{\to}X\stackrel{\pi}{\to} C.
\end{eqnarray*}
Since $\phi=\pi\circ i:B\to C$ is a finite morphism between smooth curves and $E$ is semistable then $\phi^{*}(E)$ is a semistable vector bundle (see \cite[Theorem 10.1.3]{lepotier}). Assume that $\pi^{*}(E)$ is not $nH$-semistable, then there exists  a subbundle $W\subset \pi^{*}(E)$ such that  
\begin{eqnarray*}
\mu_{nH}(W)>\mu_{nH}(\pi^{*}(E)).
\end{eqnarray*}
Notice that $W|_B$ is a subbundle of $\pi^{*}(E)|_B:=\phi^{*}(E)$ and the slopes are given by $\mu(W|_B)=\mu_H(W)$ and $\mu(\pi^{*}(E)|_B)=\mu_{nH}(\pi^{*}(E)).$ Therefore $\phi^{*}(E)=\pi^{*}(E)|_B$ is not a semistable vector bundle on $B$, which is a contradiction. It follows that $\pi^{*}(E)$ is a $nH$-semistable vector bundle. Hence  $\pi^{*}(E)$ is a $H$-semistable vector bundle (see Proposition \ref{Delta}).
\end{proof}

Next we prove that the pullback of rank two stable bundles are not just semistable but stable:

\begin{Theorem} \label{TheoremA}
Let $\pi:X\longrightarrow C$ be a fibration with reduced fibers. Then for any stable rank $2$ vector bundle $E$ over $C$, the pullback $\pi^*(E)$ is a $H$-stable bundle on $X$. 
\end{Theorem}

\begin{proof}
From the previous result we can assume that $\pi^*(E)$ is a strictly $H$-semistable vector bundle. Let $V_1\subset \pi^{*}(E)$ be a line subbundle with $\mu_H(V_1)=\mu_H(\pi^{*}(E))$.  Hence there exists an exact sequence of torsion free sheaves
\begin{equation} \label{1}
    0 \longrightarrow V_1 \longrightarrow \pi^*(E) \longrightarrow V_2 \otimes I_Z \longrightarrow 0,
\end{equation}
where $V_i$ is a line bundle on $X$ and $Z \subset X$ has codimension $2$. Restricting the exact sequence (\ref{1}) to a generic fiber $f$ such that $\text{Supp}(Z) \cap f = \emptyset$, 
\begin{equation} \label{2}
    0 \longrightarrow V_1|_f \longrightarrow \pi^*(E)|_f \longrightarrow V_2|_f \longrightarrow 0.
\end{equation}
By Lemma \ref{restricciones} part (ii), we have that $\pi^*(E)|_f = \mathcal{O}_f\oplus \mathcal{O}_f$ is semistable of degree $0$ and   $\text{deg} \,(V_1|_f) \leq 0$.

\noindent{\bf We claim that:} $\text{deg} \, (V_1|_f) < 0$. Suppose that $\text{deg} \, (V_1|_f) = 0$. From the exact sequence (\ref{2}) for a generic fiber $f$, we have
 \begin{eqnarray*}
 \text{deg} \, (V_1|_f) = \text{deg} \, (V_2|_f) =0.
\end{eqnarray*}
Notice that from Lemma \ref{seccciones0}, $h^0(V_i|_f) \leq rk(V_i)= 1$. By taking cohomology in the exact sequence (\ref{2}),
\begin{eqnarray*}
2=h^0(\pi^{*}(E)_f)\leq h^0(V_1|f)+h^0(V_2|f)\leq 1+1=2.
\end{eqnarray*}
 Therefore  $h^0(V_i|f)=1$ and $V_i|_f = \mathcal{O}_f$ for a generic fiber $f$. By Lemma \ref{principal}, it follows that there exists a line bundle $L_1$ over $C$ such that $\pi^{*}(L_1)=V_1$.
 Therefore $L_1\subset  E$ is a subbundle of $E$ and 
 \begin{eqnarray*}
  \mu(L_1)(H.f) = \mu_H(V_1)= \mu_H(\pi^*E) = \mu(E)(H.f).
 \end{eqnarray*}
 Since $H$ is ample and $f$ is a fiber, it follows that $\mu(L_1)=\mu(E)$ which contradicts the stability of $E$. Thus $\text{deg}(V_1|f) < 0$ which proves the claim. 
  We have proved the following statement: for any line subbundle $W\subset \pi^{*}(E)$ with $\mu_H(W)=\mu_H(\pi^{*}(E)),$ we have $\text{deg}(W|_f)=c_1(W)f<0$ for generic fiber $f$.
 
Let $n$ be a positive integer. Since $H$ is ample and $f$ is nef then $H_n=H+nf$ is ample.

\noindent {\bf We claim that:}  $\pi^{*}(E)$ is $H_n$-stable. Let $W\subset \pi^{*}(E)$ be a subbundle.
\begin{itemize}
\item[(i)]If $\mu_H(W)=\mu_H(\pi^{*}(E))$, by the above statement we have $c_1(W).f<0$ and 
\begin{eqnarray*}
\mu_{H_n}(W)<\mu_H(\pi^{*}(E))=\mu_{H_n}(\pi^{*}(E)).
\end{eqnarray*}
\item[(ii)]If $\mu_H(W)<\mu_H(\pi^{*}(E))$,  then
\begin{eqnarray*}
\mu_{H_n}(W)\leq \mu_{H}(W)<\mu_H(\pi^{*}(E))=\mu_{H_n}(\pi^{*}(E)).
\end{eqnarray*}

\end{itemize}

 Hence $\pi^*(E)$ is a $H_n$-stable bundle. Since $\Delta(\pi^*(E))=0$ and  the stability is independent of the chosen polarization (see Proposition \ref{Delta}), it follows that $\pi^*(E)$ is a $H$-stable vector bundle for any polarization $H$ on $X$.
 \end{proof}

We formulate our main result:

\begin{Theorem}\label{morfismobrillnoether} Let $\pi:X\to C$ be a fibration with reduced fibers. Then $\pi$ induces an injective morphism of moduli spaces
\begin{gather}\label{morfismo}
\pi^{*}:\mathcal{M}_C(2,d) \to  \mathcal{M}_{X,H}(2,df,0)\ \ \ \\
\nonumber E \mapsto  \pi^{*}E.
\end{gather}
\end{Theorem}

\begin{proof}
The proof is similar to (\cite{misra}, Theorem 5.1). First, we define the moduli functor $$\pi^{*}: \mathcal{M}(2,d) \to \mathcal{M}_{X,H}(2,df,0).$$ Let $F$ be a family of stable vector bundles of rank two and degree $d$ over $C$, parametrized by $T$. That is, $F$ is a vector bundle over $C\times T$ such that for  any closed point $t\in T$ we have that $F|_{X\times \{t\}}$ is a stable vector bundle over $C$ of degree $d$ and rank two. By Theorem \ref{TheoremA}, $\tilde{F}:=(\pi \times id)^{*}(F)$ defines a family of $H$-stable bundles of rank two over $X$ parametrized by $T$ such for any closed point $t\in T$ we have that $c_1(\tilde F|_{X\times\{ t\}})=df$ and $c_2(\tilde F|_{X\times\{ t\}})=0.$ Thus, we get a natural transformation of functors
\begin{equation*}
\pi^{*}: \mathcal{M}(2,d) \to \mathcal{M}_{X,H}(2,df,0).
\end{equation*}
Now we prove the injectivity of $\pi^{*}$: For $i=1,2$ consider $E_i\in \mathcal{M}(r,d)$ such that $\pi^{*}(E_1)\cong \pi^{*}(E_2).$ Since $\pi$ is a fibration, it follows that
\begin{eqnarray*}
E_1=\pi_{*}(\pi^{*}(E_1)) \cong \pi_{*}(\pi^{*}(E_2))=E_2. 
\end{eqnarray*}
Thus $E_1\cong E_2$ which proves the injectivity of $\pi^{*}$ and completes the proof of the Theorem.
\end{proof}

Let  $\pi:X \rightarrow C$ be a fibration with reduced fibers. If $V|_f$ is semistable for any stable vector bundle, we conjecture that the image of the morphism (\ref{morfismo}) for the case of rank $r$ stable bundles consists of $H$-stable vertical vector bundles over $X$ such that the restriction to the general fiber is trivial (see \cite[Section 2 and Definition 2.1]{Varma} for the definition of vertical bundles and see \cite[Lemma 1.4]{Bauer} for more details).

\begin{Remark}\label{isoregladaselipticas}
\begin{itemize}
\item  [(i)] \emph{ If  $\pi:X \rightarrow C$ is a ruled surface the morphism (\ref{morfismo}) is in fact an isomorphism between the respective moduli spaces of (semi)stables vector bundles for any rank $r\geq 1$ (see e.g. \cite[Corollary 4.2]{misra}).}
\item[(ii)] \emph{Let $\pi:X\longrightarrow C$ be a non-isotrivial relatively minimal elliptic fibration with no multiple fibers. Then, for any rank $r\geq 1$ the morphism (\ref{morfismo}) is an isomorphism. (For a deeper discussion of the isomorphism we refer the reader to \cite{Bauer}, \cite{FriedmanElliptic} \cite{Takemoto2} and \cite{Varma}.)}
\end{itemize}
\end{Remark}

\begin{Corollary} \label{Modelem}
Let $\pi:X \rightarrow C$  be a fibration with reduced fibers. If the moduli space $M_C(2,d)$ is non-empty, then  the moduli spaces $M_{X,H}(2,df,c_2)$ and $M_{X,H}(2,df+2c_1(L),c_2+df \cdot c_1(L)+ c_1(L)^2)$ are non-empty for any $L \in Pic(X)$ and $c_2 \geq 0$.
\end{Corollary}

The proof of corollary makes use of the following result.

\begin{Lemma} \label{coshui} (cf. \cite[Lemma 2.7]{Coskun-Huizenga})
Let $L$ be a line bundle on a smooth surface $X$.  Let $E$ be a vector bundle on $X$, and let $E'$ be a general elementary modification of $E$ at a general point $p \in X$, defined as the kernel of a general surjection $\phi:E \longrightarrow \mathcal{O}_p$:
\[0 \rightarrow E' \rightarrow E \rightarrow \mathcal{O}_p \rightarrow 0.\]
\begin{itemize}
    \item [(i)] $rk \, (E') = rk \, (E), \,\, c_1(E')=c_1(E), \, \, c_2(E')=c_2+1.$ 
    \item [(ii)] If $E$ is $H$-stable, then $E'$ is $H$-stable.
    \item [(iii)] $H^2(X,E) \cong H^2(X,E')$.
    \item [(iv)] If $h^0(X,E)>0$, then $h^0(X,E')=h^0(X,E)-1$  and $h^1(X,E')=h^1(X,E)$. If $h^0(X,E) = 0$, then $h^1(X,E') = h^1(X,E) + 1$. In particular, if at most one of $h^0$ or $h^1$ is non zero for $E$, then at most one of $h^0$ or $h^1$ is non zero for $E'$.
\end{itemize} 
\end{Lemma}

\textbf{Proof of Corollary \ref{Modelem}}.  Let $E \in M_{X,H}(2,df,0)$ and let  $E'$ be a general modification of $E$.  By Lemma \ref{coshui}, it follows that $E' \in M_{X,H}(2,df,1)$. Repeated application of elementary modification enables us to conclude that the moduli space $M_{X,H}(2,df,c_2)$ is non-empty for any $c_2 \geq 0$.  Moreover, since $E \otimes L$ and $E' \otimes L$ are $H$-stable, we can conclude that the moduli space  $M_{X,H}(2,df+2c_1(L),c_2+df \cdot c_1(L)+ c_1(L)^2)$ is non-empty for any $L \in Pic(X)$. 

 \section{Non-emptiness of Brill-Noether loci on fibered surfaces}
 
Moduli spaces of stable vector bundles have been extensively studied; however, relatively little is known about their geometry in terms of the existence and structure of their subvarieties. In \cite{laurayrosaBN}, Costa and Mir\'o-Roig have defined and constructed the Brill-Noether locus $W_{X,H}^k(r,c_1,c_2)$ for smooth projective surfaces (in fact these loci were construted for smooth projective varieties) as subvarieties of $\mathcal{M}_{X,H}(r,c_1,c_2)$ satisfying cohomological properties, i.e. the support is the set of $H$-stable rank $r$ vector bundles $E$ on $X$ with fixed Chern classes $c_i \in H
^{2i}(X,\mathbb{Z})$ for $i=1,2$ and $h^0(E)+h^2(E)\geq k$, i.e
\begin{eqnarray*}
W_{X,H}^k(r,df,0):=\{ E\in \mathcal{M}_{X,H}(r,c_1,c_2) | h^0(E)+h^2(E)\geq k \}.
\end{eqnarray*}
Moreover, if $h^2(E)=0$  for any vector bundle $E\in \mathcal{M}_{X,H}(r,c_1,c_2)$ then each non-empty irreducible component of $W_{X,H}^k(r, c_1,c_2)$ has dimension at least the Brill-Noether number defined as
\begin{equation*}
\rho_X(r,c_1,c_2,k) := dim \mathcal{M}_{X,H}(r,c_1,c_2)- k (k - \chi(r,c_1,c_2)).
\end{equation*}
The study of these subvarieties is known as Brill-Noether theory. Basic questions concerning non-emptiness, conectedness, irreducibility, dimension, singularities, etc, have been answered when $X$ is a curve (see for instance \cite{ACGH},  \cite{monserratBNrango2} and \cite{monserratrango2detfijo}). 

 In this section we use the morphism \ref{morfismo} to study properties of non-emptiness, connectedness, irreducibility and dimension of the Brill-Noether locus $W_{X,H}^k(r,df,0)$ over a fibered surface $\pi: X\longrightarrow C$ with reduced fibers. From Proposition \ref{Delta}, since stability does not depend on the polarization we will denote by $W_X^k(2,df,0)$ the Brill-Noether locus and by $\mathcal{M}_X(2,df,0)$ the moduli space.
The following theorem states a relation between the locus $W^k_C(r,d)$ and $W^k_{X}(r,df,0)$. 

\begin{Theorem} \label{App}
Let $r=1,2$ and $\pi:X\longrightarrow C$ be a fibration with reduced fibers. The morphism induced by the pullback 
\begin{equation}\label{MorApp}
\pi^{*}:W_C^k(r,d)\longrightarrow W^k_{X}(r,df,0)
\end{equation}
is injective. In particular, if $\rho_C(1,d,k+1)\geq 0$ then $W_X^k(1,df,0)\neq \emptyset$.   
\end{Theorem}

\begin{proof}
 Let $E\in W_C^k(r,d)$. From Theorem \ref{morfismobrillnoether} it is sufficient to prove that $\pi^{*}(E)\in W^k_{X}(r,df,0).$
 Since $\pi$ is a fibration, by the projection formula it follows that  $\pi_{*}(\pi^{*}(E))=E$. Therefore $h^0(X,\pi^{*}(E))= h^0(C,E)=k$  and $\pi^{*}(E)\in W_{X}^k(r,df,0)$ as desired. The second part follows directly from Theorem \cite[Theorem 1.1]{ACGH}. 
 \end{proof}
 
 \begin{Corollary}
 Let $r=1,2$ and $0 < c_2 < k$.  Let $\pi:X\longrightarrow C$ be a fibration with reduced fibers. 
 \begin{itemize}
     \item [(i)]  If the locus  $W^k_C(r,d)$ is non-empty, then the locus $W^{k-c_2}_X(r,df,c_2)$ is non-empty.     
     \item [(ii)] Let $D$ be an effective divisor on $X$.  If the locus  $W^k_C(r,d)$ is non-empty, then the locus $W^{k-c_2}_X(r,df+rc_1(\mathcal{O}_X(D)),c_2+df \cdot c_1(\mathcal{O}_X(D)+ c_1(\mathcal{O}_X(D)^2)$ is non-empty.
 \end{itemize}
 \end{Corollary}
 
 \begin{proof}
 \begin{itemize}
     \item [(i)] The proof follows directly from Theorem \ref{App} and repeated application of Corollary \ref{Modelem} and Lemma \ref{coshui}.
     \item [(ii)] The proof follows from item (i) and the exact sequence
     \[0 \rightarrow E \rightarrow E(D) \rightarrow E_D(D) \rightarrow 0.\]
 \end{itemize}
  
 \end{proof}
 
Non-emptiness of the  Brill-Nother locus $W_C^k(r,d)$, $r \geq 2$ has been studied by several authors (see for instance \cite{ivonamonserrat}, \cite{monserratBNrango2}, \cite{montserratdeterminantecanonico}, \cite{monserratrango2detfijo}, \cite{Sundaram}). However, the problem in the general case remains open. There are examples where the expected dimension $\rho_C(r,d,k) < 0$ and $W_C
^k(r,d)$ is non-empty (see for instance \cite{ivonamonserrat}); and examples where $\rho_C(r,d,k) > 0$ and $W_C
^k(r,d)$ is non-empty of dimension strictly greater than $\rho_C(r,d,k)$ (see for instance \cite[Corollary 1.2]{montserratdeterminantecanonico}). The main interest of Theorem \ref{App} is that it allows to use results of Brill-Noether over curves to determine properties of $W^k_{X}(r,df,0)$ for $r=1,2$ as we will see in the next results.
 
\begin{Corollary}
Let $\pi:X\longrightarrow C$ be fibration with reduced fibers. If any $L\in Pic^{df}(X)$ satisfies $h^2(L)=0,$ then
\begin{enumerate}
    \item[(i)] $W^k_C(1,d)\cong W^k_{X}(1,df,0)$ for any $k\in \mathbb{N}.$
        
    \item[(ii)] If $\rho_C(1,d,k+1)\geq 1$, then $W_X^k(1,df,0)$ is connected.

    \item[(iii)] If $\rho_C(1,d,k+1)< 0$ and $C$ is a general curve, then $W_X^k(1,df,0)= \emptyset$.

    \item[(iv)] If $\rho_C(1,d,k+1)\geq 1$ and $C$ is a general curve, then $W_X^k(1,df,0)$ is irreducible.
    \end{enumerate}
\end{Corollary}

\begin{proof}
We only prove $(i)$, the proof $(ii)-(iv)$ follows as an application of $(i)$ and the results well known on the classical Brill-Noether theory (see for instance \cite[Chapter 5]{ACGH}).  We claim that  $\pi^*$ is surjective.  Let $L\in  W_X^k(1,df,0)$ be  a line bundle over $X$. Since $h^2(L)=0$, it follows that $h^0(L)\geq k\neq 0$. By Lemma \ref{principal},  we recall that if $L|_F=\mathcal{O}_F$ for generic fiber $F$ then there exists a line bundle ${\bar L}$ over $C$ such that $\pi^{*}({\bar L})=L$.  Assume that $L|_F$ is not isomorphic to $\mathcal{O}_F$ for generic fiber $F$, then $h^0(F,L|_F)=0$. From the exact sequence
\begin{eqnarray*}
0\longrightarrow L(-F)\longrightarrow L\longrightarrow L|_F\longrightarrow 0,
\end{eqnarray*}
we get that $h^0(X,L(-F))=h^0(X,L)$. So $h^0(X,L)=0$, which is a contradiction. Therefore $L|_F=\mathcal{O}_F$ and there exists a line bundle ${\bar L}$ over $C$ such that $\pi^{*}({\bar L})=L.$ Since $\pi$ is a fibration, it follows that $h^0(\bar L)=h^0(L)\geq k$. Hence, by Theorem \ref{App}  the morphism 
  \begin{equation*}
      \pi^{*}:W^k_C(1,d)\longrightarrow W^k_{X}(1,df,0)
  \end{equation*}
 is bijective.  
 Since  $W^k_C(1,d)$ and $W^k_X(1,df,0)$ are normal varieties, it follows that $\pi^{*}$ is an isomorphism, which is the desired conclusion.
\end{proof}

 


\begin{Proposition}  Let $\pi: X\to C$ be a ruled surface. There exist an isomorphism  $W_{X}^k(r;df,0)\cong W_C^k(r,d)$. Moreover, in this case $\rho_C(r,d,k) = \rho_X(r,df,0,k)$.
\end{Proposition}
\begin{proof}
Since the morphism (\ref{morfismo}) is in fact an isomorphism between the respective moduli spaces of (semi)stables vector bundles for any rank $r\geq 1$ (see \cite[Corollary 4.2]{misra}), it is sufficient to prove that $h^2(E)=0$ for any $E\in M_X(r; df,0).$  We recall that there is a well defined  invariant $e$ (see \cite[Proposition 2.8]{hartshorne}). We denote by $C_0$ the section of self-intersection $-e$ and by $f$ the class of a fiber. Let $H$ be a line bundle numerically equivalent to $C_0+bf$ with $b>\text{max}\{e,e+g-1-\frac{er+d}{2r}\}, $ then $H$ is ample  (see \cite[Proposition 2.20]{hartshorne}). Since \[b> e+g-1-\frac{er+d}{2r} ,\] 
it follows that $c_1(E^{\vee}\otimes K_X).H<0$.  Therefore $H^2(X,E)=0$ for any  vector bundle $E \in \mathcal{M}_{X}(r,df,0)$ because stability is independent of the polarization $H$ (see Theorem \ref{Delta}) and Serre duality Theorem.
\end{proof}

\begin{Proposition}
Let $\pi:X\rightarrow C$ be a relatively minimal elliptic fibration with no multiple fibers and $\chi=\chi(\mathcal{O}_X)>0$. Let  $r,d\in \mathbb{N}$, satisfying $d\geq r(2(g-1)+\chi)$, then 
\[ W^k_X(r,df,0) = \begin{cases}
     \emptyset, & \text{if $k> d+r(1-g)$}, \\
     \mathcal{M}_C(r,d), & \text{if $k\leq d+r(1-g)$.}
\end{cases}\]
In the case $d\leq 2r(g-1)$ we have $W_X^k(r,df,0)\cong W_C^k(r,d)$, however bundles $[E]\in W_X^k(r,df)$ not necessary satisfy $h^2(E)=0$.
\end{Proposition}

\begin{proof}
Since the morphism (\ref{morfismo}) is in fact an isomorphism between the respective moduli spaces of (semi)stables vector bundles for any rank $r\geq 1$ (see Remark \ref{isoregladaselipticas} part (ii)), it is sufficient to prove that $h^2(E)=0$ for any $E\in M_X(r; df,0).$ Since   $h^1(\mathcal{O}_f)=1$ for every fiber of $\pi,$  standard base change results imply that $R^1\pi_{*}(\mathcal{O}_X)$ is a line bundle on $C$. Denote by $L$ the dual line bundle and  $deg(L)=\chi$. By \cite[Theorem 15]{Friedman}, the canonical line bundle for an elliptic surface is given by
    \begin{eqnarray*}
    K_X=\pi^{*}(K_C\otimes L).
    \end{eqnarray*}
     If $(2g-2+\chi)r<d$,  then $h^2(F)=0$ for any $F\in M_X(r,df,0)$  and  $\pi^{*}$ induce an isomorphism between $W_C^k(r,d)$ and $W_X^k(r;df,0)$ for any $k\in \mathbb{N}$.
     Let $E\in M_C(r,d)$ be a vector bundle, we have that $h^0(E)=d+r(1-g)$ by Riemann-Roch Theorem, and $h^0(\pi^{*}(E))=d+r(1-g)$. Thus,  
\[ W^k_X(r,df,0) = \begin{cases}
     \emptyset, & \text{if $k> d+r(1-g)$}, \\
     \mathcal{M}_C(r,d), & \text{if $k\leq d+r(1-g)$.}
\end{cases}\]
Now, the Brill-Noether number satisfies
\begin{eqnarray*}
\rho_X(r,df,0,k) &:=& r^2(g-1)+1- k (k - r\chi(\mathcal{O}_X))\\
&=& r^2(g-1)+1+\frac{r^2}{4}\chi^2(\mathcal{O}_X) - (k-\frac{r}{2}\chi(\mathcal{O}_X))^2.
\end{eqnarray*}
Thus, $\rho_X<0$ if $k>\sqrt{ r^2(g-1)+1+\frac{r^2}{4}\chi^2(\mathcal{O}_X)}+\frac{r}{2}\chi(\mathcal{O}_X).$ In particular, if $k=d+r(g-1)$ then $\rho_X<0.$
\end{proof}

A similar argument proves that if $\pi: X\to C$ is a fibration with reduced fibers and $K_X$ is ample with $rK_X^2\leq dK_X.f=2d(g(f)-1)$ then $H^2(X,E)=0$ for any $E \in \mathcal{M}_X(2,df,0)$ and $\rho_X(2,df,0,k)\leq \rho_C(2,d,k)$.

\section{Emptiness of Brill-Noether loci of rank two stable vector bundles on surfaces}

In this section we give a generalization of Clifford Theorem for rank $2$ vector bundles on surfaces and  show the emptiness of some Brill-Noether loci.

\begin{Proposition}\label{rankone} Let $X$ be a smooth projective surface and let $H$ be a very ample divisor on $X$ such that $H.K_X\geq 0$. Let $\mathcal{F}$ be a free torsion sheaf of rank one over $X$ such that $0\leq c_1(\mathcal{F}).H\leq nH^2$ with $n\in \mathbb{N}$ and $c_1(\mathcal{F})\notin |nH|$. Then
\[
 h^0(\mathcal{F})\leq  n\frac{c_1(\mathcal{F}). H}{2}+1.
 \]
\end{Proposition}

\begin{proof} Let $C$ be a smooth projective curve in the linear system  $|nH|$.  Assume that $\mathcal{F}:=L$ is a line bundle over $X$ such that $0 \leq L . H \leq  nH^2$ and $L\notin |nH|$. 
By Adjunction formula and the fact that $H.K_X\geq 0$, we have  that $deg(\mathcal{O}_C(L))\leq n^2 H^2\leq C^2+K_X.C =2(g(C)-1)$. Therefore,  
\[ h^0(\mathcal{O}_C(L)) \leq  n\frac{L. H}{2}+1\] 
by Clifford's Theorem.

On the other hand, since $H$ is an ample line bundle and $L.H\leq nH^2$ it follows from Nakai-Moishezon's criterion that  $H^0(\mathcal{O}_X(L-nH))=H^0(L-C)=0$. 

From the exact sequence
\begin{eqnarray*}
0\to \mathcal{O}_X(L-C)\to \mathcal{O}_X(L)\to \mathcal{O}_C(L)\to 0,
\end{eqnarray*}
 it follows that 
\[ h^0(L) \leq  n\frac{L . H}{2}+1.\]  
 which proves the theorem for line bundles. If $\mathcal{F}$ is a torsion free sheaf of rank one, then there exist a line bundle $L$ over $X$ such that $\mathcal{F}=L\otimes I_Z$, where $Z\subset X$ is of codimension 2. Since $c_1(\mathcal{F})=c_1(L)$ and $h^0(\mathcal{F})\leq h^0(L)$ it follows that 
 \[
 h^0(\mathcal{F})\leq  h^0(L) \leq n\frac{c_1(\mathcal{F}). H}{2}+1.
 \]
\end{proof}
 
The following theorem can be considered as a generalization of Clifford's Theorem.

\begin{Theorem} \label{ranktwo}
Let $X$ be a smooth projective surface and let $H$ be a  very ample divisor on $X$ such  that $H\cdot K_X \geq 0$. Let $E\in M_{X,H}(2,c_1,c_2)$ with $0\leq \mu_H(E) < nH^2$. Then
\[h^0(E) \leq n\frac{c_1 . H}{2}+n^2H^2 +2.\]
\end{Theorem}
 
\begin{proof}
Notice that if $h^0(E)=0,$ then the theorem follows. Assume that $h^0(E)\neq 0.$ Let $L_1$ be a subline bundle of $E$ of maximal slope.  Since $h^0(E)\neq 0$, it follows that $0 \leq L_1.H < nH^2$. By Proposition \ref{rankone},
\[ h^0(L_1) \leq n \frac{L_1. H}{2}+1.\] 
Consider the exact sequence
 \begin{eqnarray}\label{Ex}
0\rightarrow L_1\rightarrow E\rightarrow L_2\otimes I_Z\rightarrow 0.
\end{eqnarray}
Since $0 \leq \mu_{H}(E) < nH^2$, from the exact sequence (\ref{Ex}) follows that   $0 \leq \mu_H(E) <  L_2.H$ and $L_2$ satisfies the hypothesis of Proposition \ref{rankone}. Hence, 
\[h^0(E) \leq n\frac{L_1. H}{2} +1+ 2n\frac{L_2.H}{2} +1 \leq n\frac{c_1 . H}{2}+n^2H^2 +2.\]
\end{proof}

Notice that an argument similar to the  proof of Proposition \ref{rankone} and Theorem \ref{ranktwo} works when $H$ is ample and there exists a smooth curve $C\in |nH|$. As an application of the Theorem \ref{ranktwo} we obtain the following result concerning the emptiness
of the Brill–Noether loci. 

\begin{Corollary}\label{empty}
Let $X$ be a smooth surface and let $H$ be a very ample divisor on $X$ such  that $H\cdot K_X \geq 0$. Let $c_2\gg 0,  n$ be integers and $c_1$ a divisor on $X$, such that $0\leq \frac{c_1.H}{2}< nH^2$. Then,
 \[W_{X,H}^k(2,c_1,c_2)= \emptyset\]
for any $k > n\frac{c_1 . H}{2}+n^2H^2 +2$.
\end{Corollary}

\begin{Remark} \emph{From \cite[Proposition 2.4]{laurayrosaregladas} whenever $c_2\gg0$ the moduli space $\mathcal{M}_{X,H}(2,c_1,c_2)$ is a non-empty generically smooth, irreducible, quasi-projective variety of the expected dimension  
$\text{dim}(\mathcal{M}_{X,H}(2,c_1,c_2))=4c_2-c_1^2-3\chi (\mathcal{O}_X)$. 
In particular, if  $c_2\gg 0$ and $k\geq 5$ in Corollary \ref{empty}, when the Brill-Noether locus is empty, the expected dimension is $\rho_{X,H}(2,c_1,c_2,k)<0.$}
\end{Remark}


\end{document}